\documentclass[12pt,reqno]{amsart}

\usepackage{latexsym}
\usepackage{amssymb}

\renewcommand{\epsilon}{\varepsilon}

\newcommand{\szego}{Szeg\H{o} }

\newcommand{\kahler}{K\"ahler }

\newcommand{\R}{{\mathbb R}}
\newcommand{\C}{{\mathbb C}}

\newcommand{\dbar}{\bar\partial}
\newcommand{\ddbar}{\partial\dbar}

\newcommand{\E}{{\mathbf E}}

\renewcommand{\phi}{\varphi}

\newcommand{\dcal}{\mathcal{D}}
\newcommand{\ecal}{\mathcal{E}}

\newcommand{\hcal}{\mathcal{H}}

\newcommand{\ocal}{\mathcal{O}}

\newcommand{\scal}{\mathcal{S}}

\newcommand{\La}{\Lambda}

\newcommand{\de}{\delta}
\newcommand{\De}{\Delta}

\newtheorem{theo}{{\sc Theorem}}[section]
\newtheorem{maintheo}{{\sc Theorem}}
\newtheorem{cor}[theo]{{\sc Corollary}}
\newtheorem{maincor}[maintheo]{{\sc Corollary}}
\newtheorem{lem}[theo]{{\sc Lemma}}
\newtheorem{prop}[theo]{{\sc Proposition}}

\newenvironment{rem}{\medskip\noindent{\it Remark:\/} }{\medskip}

\title[Real and complex zeros of Riemannian random waves]
{Real and complex zeros of Riemannian random waves }

\author{Steve Zelditch }
\address{Department of Mathematics, Johns Hopkins University, Baltimore, MD
21218, USA}

\thanks{Research  partially supported by NSF grant  DMS-0603850.}

\date{\today}

\begin{document}

\begin{abstract} We show that the expected limit distribution of
the real zero set of a Gaussian random linear combination of
eigenfunctions with frequencies from a short interval
(`asymptotically fixed frequency') is uniform with respect to the
volume form of a compact Riemannian manifold $(M, g)$. We further
show that the complex zero set of the analytic continuations of
such Riemannian random waves to a Grauert tube in the
complexification of $M$ tends to a limit current.

\end{abstract}

 \maketitle

This article is concerned with the real and complex zero sets  of
{\it Riemannian random waves} on a real analytic Riemannian
manifold $(M, g)$.   To define Riemannian random waves, we  fix an
orthonormal basis $\{\phi_{\lambda_j}\}$ of real-valued
eigenfunctions of the Laplacian $\Delta_g$ of $(M, g)$,
$$\Delta_g \phi_{\lambda_j} = \lambda_j^2 \phi_{\lambda_j},
  \;\; \langle \phi_{\lambda_j}, \phi_{\lambda_k}
\rangle = \delta_{jk}, $$ and define Gaussian ensembles of
 random functions   $f = \sum_j c_j \phi_{\lambda_j}$ of the
 following two types:
\begin{itemize}

\item The asymptotically  {\it fixed frequency} ensemble
$\hcal_{I_{\lambda}}$,  where $I_{\lambda} = [\lambda, \lambda +
1]$ and where  $\hcal_{I_{\lambda}}$ is the vector space of linear
combinations
\begin{equation} \label{psi} f_{\lambda} = \sum_{j: \lambda_j
\in [\lambda, \lambda + 1]} c_j\;\; \phi_{\lambda_j},
\end{equation}  of eigenfunctions  with $\lambda_j$ (the
frequency) in an interval $[\lambda, \lambda + 1]$ of fixed width.
(Note that it is the square root of the eigenvalue of $\Delta$,
not the eigenvalue,  which is asymptotically fixed).

\item The {\it cut-off} ensembles $\hcal_{[0, \lambda]}$ where the
frequency is cut-off at $\lambda$:
\begin{equation} \label{psi2} f_{\lambda} = \sum_{j: \lambda_j
\leq \lambda } c_j\;\; \phi_{\lambda_j},
\end{equation}

\end{itemize}

 By {\it random}, we mean  that the coefficients $c_j$ are
independent Gaussian random variables with mean zero and with the
variance defined so that the expected $L^2$ norm of $f$  equals
one. Equivalently, the real vector spaces $\hcal_{[0, \lambda]}, $
resp. $\hcal_{I_{\lambda}}$ are endowed with the inner product
$\langle u, v \rangle = \int_M u v dV_g$ (where $dV_g$ is the
volume form of $(M, g)$) and random means that we equip the vector
spaces  with the induced Gaussian measure. Our main results given
the asympototic distribution of real and complex zeros of such
Riemannian random waves in the high frequency limit $\lambda \to
\infty$.

 The real zeros are straightforward to define. For each $f_{\lambda} \in \hcal_{[0, \lambda]}$ or
 $ \hcal_{I_{\lambda}}$ we
 associated to
the zero set $Z_{f_{\lambda}} = \{x \in M: f_{\lambda} (x) = 0\}$
 the positive measure
\begin{equation} \label{LINSTAT} \langle |Z_{f_{\lambda}}|, \psi \rangle = \int_{Z_{f_{\lambda}}}
\psi d\hcal^{n - 1}, \end{equation} where $d\hcal^{m-1}$ is the
induced (Hausdorff) hypersurface measure. Our  first result
(Theorem \ref{ZMEASURE}) shows that the  normalized expected limit
distribution $\frac{1}{\lambda} \E |Z_{f_{\lambda}}|$ of zeros of
random Riemannian waves  tends to the volume form $dV_g$  as
$\lambda \to \infty$. The result is the same in the cutoff and
fixed frequency ensembles, but the proof is simpler in the former.
Further, in the fixed frequency case, if one chooses a random
sequence
 $f_{\lambda_N}$, i.e. the elements  are chosen independently from the
frequency intervals $[N, N + 1]$ and at random from $\hcal_{[N, N
+ 1]}$, then almost surely, $\frac{1}{N} \sum_{k = 1}^N
\frac{1}{\lambda_k} |Z_{f_{\lambda_k}}| \to dV_g. $

The  {\it complex zeros} are defined by analytic continuation.
Since $(M, g)$ is assumed to be real analytic, it admits a
complexification $M_{\C}$ and  the eigenfunctions can be
analytically continued to  $M_{\C}$. Their real zero hypersurfaces
extend to  complex nodal hypersurfaces in $M_{\C}$. Our second
result (Theorem \ref{ZERORAN}) determines the limit distribution
of these random complex nodal hypersurfaces, and  shows that the
limit distribution is the same as the one determined in \cite{Z3}
for complex zeros of analytic continuations of ergodic
eigenfunctions, e.g. eigenfunctions of $\Delta_g$  when the
geodesic flow of $(M, g)$ is ergodic. This corroborates the random
wave hypothesis as applied to complex nodal sets, i.e. the
conjecture  that eigenfunctions of chaotic systems resemble random
waves. Note that
 in the real domain, the
distribution of nodal lines of eigenfunctions in the ergodic case
is very much an open problem. The motivation to study  complex zeros
of analytic continuations of eigenfunctions comes from the fact one has more control
over complex zeros than real zeros of (deterministic) eigenfunctions.

\subsection{Statement of results on real zeros}

We now state our results more precisely.  Let us first consider
the real zero sets of eigenfunctions of $\Delta$ for the  standard
sphere $(S^m, g_0)$. Let ${\mathcal H}_N \subset L^2(S^m)$ denote
the real $d_N$-dimensional  inner product  space  of spherical
harmonics of degree $N$. The eigenvalue is   given by
$$(\lambda_N^{S^m})^2 = N (N + m - 1) = (N + \frac{\beta}{4})^2  - (\frac{\beta}{4})^2$$ where
$\beta = \frac{m - 1}{2}$ is the common Morse index of the $2 \pi$
periodic geodesics and $$d_N = { m  + N - 1 \choose N}
  - { m + N - 3 \choose N -
2}.$$

We choose an orthonormal basis $\{\phi_{N j}\}_{j = 1}^{d_N}$ for
$\hcal_N$. For instance, on $S^2$ one can choose the real and imaginary parts of the standard
$Y^N_m$'s. In the fixed frequency ensemble, we  endow the real
vector space $\mathcal{H}_N$ with the Gaussian probability measure
 $\gamma_N$ defined by
\begin{equation}\label{gaussian}\gamma_N(f)=
\left(\frac{d_N}{\pi }\right)^{d_{N}/2}e^ {-d_{N}|c|^2}dc\,,\qquad
f=\sum_{j=1}^{d_{\lambda}}c_j \phi_{N j}, \,\;\; d_{N} = \dim
\hcal_{N}.
\end{equation} Here,  $dc$
is $d_{N}$-dimensional real Lebesgue measure. The normalization is
chosen so that $\E_{\gamma_N}\; \langle f,f\rangle=1$, where
$\E_{\gamma_N}$ is the expected value with respect to $\gamma_N$.
Equivalently,  the $d_N$ real variables $\ c_j$
($j=1,\dots,d_{N}$) are independent identically distributed
(i.i.d.) random variables with mean 0 and variance
$\frac{1}{2d_{N}}$; i.e.,
$$\E_{\gamma_N} c_j = 0,\quad  \E_{\gamma_N} c_j c_k =
\frac{1}{2 d_N}\de_{jk}\,.$$ We note  that the Gaussian ensemble
is equivalent to picking $f_N \in \hcal_{N}$ at random from the
unit sphere in $\hcal_{N}$ with respect to the $L^2$ inner
product. The latter description is more intuitive but it is
technically more convenient to work with Gaussian measures. In the
cutoff ensemble, we put the  product  Gaussian measure $\Pi_{n =
1}^N \gamma_n$ on $\bigoplus_{n = 1}^N \hcal_n$.

We now consider the analogous constructions on a general compact
Riemannian manifold $(M, g)$ of dimension $m$.  As mentioned
above, and as defined more precisely in \S \ref{RIEMRW},   the
analogue of the space $\hcal_N$ of spherical harmonics of degree
$N$ is played by the space $\hcal_{I_{N}}$ of linear combinations
of eigenfunctions (\ref{psi}) with frequencies in an interval $I_N
: = [N, N + 1]$. The precise decomposition of $\R$ into intervals
is not canonical on a generic Riemannian manifold and the results
do not depend on the choice. We choose $I_N = [N, N + 1]$ only for
notational simplicity.  In the special case of Zoll manifolds (all
of whose geodesics are closed), there is a canonical choice (an
eigenvalue cluster decomposition) which is described in \S
\ref{RIEMRW}). Henceforth we abbreviate $\hcal_N = \hcal_{I_N}$ on
general Riemannian manifolds. We continue to denote by $\{\phi_{N
j} \}_{j = 1}^{d_N}$ an orthonormal basis of $\hcal_N$  where $d_N
= \dim \hcal_N$.   We equip it with the Gaussian measure
(\ref{gaussian}) and again denote the expected value with respect
to $(\hcal_N, \gamma_N)$ by $\E_{\gamma_N}$.

Our main result on real Riemannian random waves is to determine
the expected value $\E_{\gamma_N} |Z_{f_N}|$. It  is a positive
measure satisfying
\begin{equation} \label{EZ} \langle \E_{\gamma_N} |Z_{f_N}|, \psi
\rangle = \E_{\gamma_N} X_{\psi}^N, \end{equation} where
$X_{\psi}^N$ is a  `linear statistic', i.e. the random variable
\begin{equation} \label{LINSTAT2} X_{\psi}^N (f_N) = \langle \psi, |Z_{f_N}| \rangle, \;\;\; \psi \in C(M) \end{equation} considered in (\ref{LINSTAT}).

\begin{maintheo}\label{ZMEASURE} Let $(M, g)$ be a compact Riemannian
manifold,let $\hcal_{[0, \lambda]}$ be the cutoff ensemble  and
let $({\mathcal H}_{N}, \gamma_{N})$ be the ensemble of Riemannian
waves of asymptotically fixed frequency. Then in either ensemble:
\begin{enumerate}

\item For any $C^{\infty}$ $(M, g)$,  $\lim_{N \to \infty}
\frac{1}{N} {\bf E}_{\gamma_N} \langle |Z_{f_N}|, \psi \rangle =
\int_M \psi dV_g$.

\item  For a real analytic $(M, g)$, $Var(\frac{1}{N} X^N_{\psi} )
) \leq C. $

\end{enumerate}

\end{maintheo}

We restrict to real analytic metrics in (2) for the sake of
brevity. In that case, the  variance estimate follows easily from
a result of  Donnelly-Fefferman on volumes of real nodal
hypesurfaces of real analytic $(M, g)$.

In the case of the standard metric $g_0$  on $S^m$, it is obvious
apriori that  $\E_{\gamma_N} |Z_{f_N}| = C_N dV_{g_0}$, the
constant $C_N$  being the expected volume of the zero sets. The
expected volume was first  determined by P. B\'erard
by a different method.  The theorem above shows that
asymptotically the same result holds on any compact Riemannian
manifold.

A much better variance estimate for $S^m$ was obtained by J.
Neuheisel in his (unpublished) Hopkins PhD thesis \cite{Ne}, which
shows that the variance tends to zero at least at a rate $N^{-
\delta}$ for a certain $\delta > 0$.
 It is very likely that one
could prove the same (or a better) variance estimate on a general
$C^{\infty}$ $(M, g)$, but that  would require a study of the pair
correlation function of zeros which would take us too far afield
from our main purpose.   We  plan to carry them out on a different
occasion. Relatively sharp variance estimates for eigenfunctions
on arithmetic tori are given in \cite{RW,ORW}.

An immediate consequence,  by the Kolmogorov strong law of large
numbers, is a limit law for random sequences of random real
Riemannian waves. By a random sequence,  we mean an element of the
product probability space
\begin{equation} \label{HINFTY} \hcal_{\infty} = \Pi_{N = 1}^{\infty} \hcal_N, \;\;
\gamma_{\infty} = \Pi_{N = 1}^N \gamma_N. \end{equation}

\begin{maincor} \label{ZDSM} Let $(M, g)$ be a compact real analytic Riemannian
manifold, and let  $\{f_{_N} \}$ be a random sequence in
(\ref{HINFTY}).  Then
$$\frac{1}{N}\sum_{n = 1}^N \frac{1}{\lambda_n}
|Z_{f_{n}}| \to dV_g \;\;\;\; \mbox{almost surely w.r.t.} \;
(\hcal_{\infty}, \gamma_{\infty}).
$$
\end{maincor}

It is natural to conjecture that $ \frac{1}{\lambda_N} |Z_{f_{N}}|
\to dV_g$  almost surely without averaging in $N$, but the proof
would again require a stronger variance estimate than we currently
possess.

\subsection{Statement of results on complex zeros}

We now turn to results on complex zeros of  analytic continuations
of eigenfunctions. By a theorem of Bruhat-Whitney \cite{BW},  an analytic
manifold $M$ admits a complexification $M_{\C}$ into which $M$
embeds as a totally real submanifold.  Associated to $g$ is a
plurisubharmonic exhaustion function $\rho(\zeta)$ which measures
the square of the distance to the real subset $M$.  The sublevel
set $M_{\tau} = \{\zeta \in M_{\C}: \sqrt{\rho}(\zeta) < \tau\}$
is known as the Grauert tube of radius $\tau$ (cf. \cite{Gr,GS1,
GS2,LS1}).

 It was observed by Boutet de Monvel \cite{Bou}
that eigenfunctions can be analytically continued to the maximal
Grauert tube as  holomorphic functions $\phi_{\lambda_j}^{\C}$.
Thus, we can complexify the Gaussian random waves as
$$f_N^{\C} = \sum_{j = 1}^{d_N} c_{N j} \phi_{N j}^{\C}. $$
We note that the coefficients $c_{N j}$ are real and that the
Gaussian measure on the coefficients remains the real Gaussian
measure $\gamma_N$.

Our next result determines the expected limit current of complex
zeros of $f_N^{\C}$. The current of integration over the complex
zero set
$$Z_{f_N^{\C}} = \{\zeta \in M_{\C}: f_N^{\C}  = 0\}
$$ is the $(1,1)$ current defined by
$$\langle [Z_{f_N^{\C}}], \psi \rangle = \int_{Z_{f_N^{\C}}} \psi, \;\; \psi \in \dcal^{m-1, m-1}(M_{\C}), $$
for smooth test forms of bi-degree $(m-1, m-1)$. In terms of
scalar functions $\psi$ we may define $Z_{f_N^{\C}}$ as the
measure,
$$\langle [Z_{f_N^{\C}}], \psi \rangle = \int_{Z_{f_N^{\C}}}  \psi
\omega_g^{m-1}/(m-1)!, $$ where $\omega_g = i \ddbar \rho$ is the
\kahler metric adapted to $g.$

\begin{maintheo}\label{ZERORAN}  Let $(M, g)$ be  a  real
analytic compact Riemannian manifold. Then for either of the
ensembles of Theorem \ref{ZMEASURE}, we have
$$\E_{\gamma_N} \left(\frac{1}{N} [Z_{f_N^{\C}}] \right) \to  \frac{i}{ \pi} \ddbar |\xi|_g,\;\;
 \mbox{weakly in}\;\; \dcal^{' (1,1)} (B^*_{\epsilon} M).  $$
\end{maintheo}
 As mentioned above, this  result shows that the  complex zeros of the random
waves have the same expected limit distribution found in \cite{Z3}
for real analytic compact Riemannian manifolds with ergodic
geodesic flow.

\subsection{The key objects in the proof}

The principal objects (for the asymptotically fixed frequency
ensembles)  are the two point functions

\begin{equation} \Pi_{I_N}(x, y) = \E_{\gamma_N} (f_N(x) f_N(y)) =
\sum_{j: \lambda_j \in I_N} \phi_{\lambda_j}(x)
\phi_{\lambda_j}(y),
\end{equation}
i.e. the spectral projections kernel for $\sqrt{\Delta}$,  and
their analytic extensions to the totally real anti-diagonal in
$M_{\C} \times M_{\C}$ defined by
\begin{equation}\label{CXSPECPROJ}  \Pi_{I_N}(\zeta, \bar{\zeta}) = \sum_{j:
\lambda_j \in I_k} |\phi_j^{\C}(\zeta)|^2.  \end{equation} A key
point is that the latter kernels are very much off the diagonal
for non-real $\zeta$, so that the kernels grow at an exponential
rate. In the cutoff ensemble, the spectral projections kernels are
replaced by
\begin{equation} \Pi_{[0, N]}(x, y) = \E_{\gamma_N} (f_N(x) f_N(y)) =
\sum_{j: \lambda_j \in [0, N]} \phi_{\lambda_j}(x)
\phi_{\lambda_j}(y),
\end{equation}
and similarly for the complexification.

  In the real domain, the distribution of zeros
of random Riemannian waves is obtained by using formalism of
\cite{BSZ1,BSZ2} to express the density of zeros in terms of the
kernels $\Pi_{I_N}(x,y)$.  We then use the spectral asymptotics of
these kernels and their derivatives to derive the limit
distribution of zeros. A more in-depth analysis of their
off-diagonal decay could give bounds on the variance, but as
mentioned above we postpone that to a later occasion. In the
complex domain, the spectral asymptotics have not been studied
before. The asymptotics are more difficult than in the real domain
and have an independent interest. On the other hand, the link
between these kernels and the zero distribution is simpler, and we
use the Poincar\'e-Lelong method of \cite{BSZ3} rather than
\cite{BSZ1,BSZ2}.

\subsection{Discussion}

This is the first article discussing zeros of Riemannian  random
waves of asymptotically fixed energy.  We digress to compare our definitions
and results on  Riemannian waves to other definitions
and results in the literature.

The subject of random polynomials and Fourier series and their
zeros is classical; see \cite{PW} for one of the classics. For a contemporary
treatment of Gaussian random functions in a geometric setting, see \cite{AT}.
  Aside from its pure mathematical interest,   Gaussian random waves have been important in
  various branches of physics. In particular, a somewhat vague heuristic principle due to
 M.V. Berry \cite{B} asserts that  random waves should be a good model for
 quantum chaotic eigenfunctions (\cite{Z1} contains rigorous results in this direction).
 Random waves in \cite{B} and in much
  of the physics literature  are  random
Euclidean plane waves of fixed energy. For a mathematician, they
are defined   by putting a Gaussian measure on the infinite
dimensional space $\ecal_{\lambda}$  of Euclidean  eigenfunctions
of fixed eigenvalue $\lambda^2$  on $\R^m$. The Gaussian measure
satisfies $E ||f||^2 = 1$, where $||\cdot||^2$  is  the inner
product  invariant under the eigenspace representation of the
Euclidean motion group.

On a compact Riemannian manifold, the closet analogue to random
Euclidean plane waves of fixed frequency  is that of random
spherical harmonics of fixed degree on $S^m$, where one now puts
the $SO(m + 1)$-invariant normalized Gaussian measure. In both
cases, the measure is defined on an eigenspace. For a generic
Riemannian manifold, the eigenspaces are of dimension one, so one
cannot define an interesting Gaussian measure on the eigenspaces.
The alternative proposed here and in \cite{Z1} is to replace
eigenspaces by the spaces $\hcal_N$ spanned by eigenfunctions with
asymptotically constant frequency. From the viewpoint of
microlocal (semi-classical) analysis, the analogy is obvious since
the spectral projections kernels have, to leading order, the same
asymptotics as those for spherical harmonics. In  \cite{Z1} the emphasis
was on random orthonormal bases
as models for an orthonormal basis of ergodic eigenfunctions; while
here we only study individual random waves.
 Riemannian random
waves of asymptotically fixed energy seem to be  a natural global model
for random waves on a Riemannian manifold without boundary,
and the set-up extends  naturally to Riemannian manifolds with
boundary and with fixed boundary conditions on $\Delta$.  When
studying local behavior, Riemannian waves resemble the Euclidean
plane waves of the same frequency. More precisely, the scaling
limit of Riemannian random waves on length scales $\lambda^{-1}$
should give back the Euclidean plane wave model with eigenvalue
$1$. This would be the Riemannian analogue of the universality
result of \cite{BSZ1}. Thus, the natural role of Euclidean plane
waves seems to be  to capture the random behavior of small length
scales of order of $\lambda^{-1}$.

Further   motivation to study nodal lines of
Riemannian random waves  has arisen in recent   conjectures that nodal
lines of random two-dimensional  Euclidean plane waves of fixed
frequency (and chaotic eigenfunctions) tend to $SL_6$ curves
\cite{BS, FGS, BGS, SS}. This behavior should be sufficiently
universal that it should hold for Riemannian random waves of
asymptotically fixed frequency on general surfaces.  A related
conjecture asserts that random nodal lines of partial sums  of the
Gaussian free field (i.e. random Riemannian waves in the cutoff
ensemble) on a two-dimensional Riemannian surface tend to $SLE_4$
curves \cite{SS}. Thus, one expects different behavior of random
Riemannian waves for (almost) fixed  frequency and for long
frequency intervals; although the SLE connection is far outside the scope
of this article, it does motivate us  to consider both
ensembles.  For recent and deep results on nodal lines of random
spherical harmonics which are related to conjectures in
\cite{BS,BGS} we refer to \cite{NS}.

The
 rationale for studying complex zeros of Riemannian random waves
  is that only in the complex domain can we rigorously compare
 the nodal sets of ergodic eigenfunctions and those of random waves.
     In addition, the complex zeros of complexified Riemannian
     random waves is a higher dimensional generalization of the  classical  ensembles of
of   Kac-Hammersley of complexified random real polynomials. They
may be viewed as sums of complexified eigenfunctions  on a circle,
although  the spectral intervals are $[0, N]$ rather than $[N - 1,
N]$. Another recent study of complex zeros of complexified real
polynomials is the thesis of B. MacDonald \cite{Mc}; however, a
significant difference is that the polynomials there are
orthonormalized in the complex domain rather than the real domain.

\section{\label{RIEMRW} Background on densities and correlations of zeros of real Gaussian random waves}

In this section, we apply  the formalism  in  \cite{BSZ2} to give
explicit formulae for the densities of zeros of Riemannian random
waves. The same formalism also could be used to give formulae for
correlations between zeros.

\subsection{Assumptions on $(M, g)$}

  We will assume the geodesic flow $G^t$ of $(M, g)$  is of  one of the
following two types:

\begin{enumerate}

\item  {\it aperiodic:} The Liouville measure of the closed
 orbits of $G^t$, i.e. the set of vectors lying on closed geodesics,  is zero; or

\item  {\it periodic = Zoll:} $G^T = id$ for some  $T>0$;
henceforth $T$ denotes the minimal period.  The common Morse index
of the $T$-periodic geodesics will be denoted by $\beta$.

\end{enumerate}
In the real analytic case, $(M, g)$ is automatically one of these
two types, since a positive measure of closed geodesics implies
that all geodesics are closed. We only need to assume $(M, g)$ is
real analytic when considering complex zeros. In the $C^{\infty}$
case, it is simple to construct examples with a positive but not
full  measure of closed geodesics (e.g. a pimpled sphere).

The two-term Weyl laws counting eigenvalues of $\sqrt{\Delta}$ are
very different in these two cases.

\begin{enumerate}

\item  In the {\it aperiodic} case, Ivrii's two term Weyl law
states
$$N(\lambda ) = \#\{j:\lambda _j\leq \lambda \}=c_m \;
Vol(M, g) \; \lambda^m +o(\lambda ^{m-1})$$
 where $m=\dim M$ and where $c_m$ is a universal constant.

\item  In the {\it periodic} case,
 the spectrum of $\sqrt{\Delta}$ is a union of eigenvalue clusters $C_N$ of the form
$$C_N=\{(\frac{2\pi}{T})(N+\frac{\beta}{4}) +
 \mu_{Ni}, \; i=1\dots d_N\}$$
with $\mu_{Ni} = 0(N^{-1})$.   The number $d_N$ of eigenvalues in
$C_N$ is a polynomial of degree $m-1$.
\end{enumerate}

We refer to \cite{Ho,SV,Z1} for background and further discussion.

\subsection{Definition of Riemannian random waves}

 To define Riemannian random waves, we partition the
spectrum of $\sqrt{\Delta_g}$ into certain  intervals $I_N$ of
width one  and denote by $\Pi_{I_N} $ the spectral projections for
$\sqrt{\Delta_g}$ corresponding to the interval $I_N$. The choice
of the intervals $I_N$ is rather arbitrary for aperiodic $(M, g)$
and as mentioned above we  assume $I_N = [N, N + 1]$. But the
choice has to be made carefully for Zoll manifolds.

In the Zoll case, we center the intervals around the center points
$\frac{2\pi}{T} N + \frac{\beta}{4}$ of the $N$th cluster $C_N$.
 We call call such a choice of
intervals a cluster decomposition. We denote by $d_N$ the number
of eigenvalues
  in $I_N$ and  put $\hcal_N = \mbox{ran} \Pi_{I_N}$ (the range of
 $\Pi_{I_N}$). Thus, $\hcal_N$ consists of
 linear combinations $\sum_{j: \lambda_j \in I_N} c_j
\phi_{N, j}$  of the eigenfunctions $\{\phi_{N j}\}$  of
$\sqrt{\Delta_g}$ with eigenvalues in $I_N$.

The formalism is simpler in the cutoff ensemble and only requires
small modifications from the asymptotically fixed frequency
ensembles, so we only explain at the end how to modify the results
in that case.

\subsection{\label{DEN} Density of real zeros}

The formula for the density of zeros of random elements of
$\hcal_{N}$ can be derived from the general formalism of
\cite{BSZ1,BSZ2,BSZ3}.

As above, we
 let $|Z_f|$ denote the Riemannian
$(m-1)$-volume on $Z_f$.   By the  general formula of \cite{BSZ1,BSZ2},

\begin{equation}\label{d2} \E|Z_{f_N}| =K_1^N (z)dV_g \,,\quad
K_1^N(x)= \int D(0,\xi,x) ||\xi|| \; d\xi\,.
\end{equation}
We digress to connect this formula with the discussions in
\cite{BSZ1,BSZ2,Ne}. In these articles $||\xi|| $ is written
$\sqrt{\det(\xi\xi^*)}.$ However,  $\det(\xi\xi^*) = ||\xi||^2$ in
the codimension one case. Indeed,  let  $df^*_x$ be the adjoint
map with respect to the inner product $g$  on  $T_x M$. Let $df_x
\circ df_x^*: T_x M \to \R$  be the composition. By $\det df_x
\circ df_x^*$ is meant the determinant with respect to the inner
product on $T_x M$; it clearly equals $|df|^2$ in the codimension
one case.

The formulae of \cite{BSZ1, BSZ2} (the `Kac-Rice' formulae) give
that
\begin{equation}\label{fg4}
D(0,\xi;z)=Z_{n}(z) D_{\La}(\xi;z),
\end{equation}
where
\begin{equation}\label{fg5}
D_{\La}(\xi;z)=\frac{1}{\pi^{m}\sqrt{\det\La}}\exp\left( -{\langle
\La^{-1}\xi,\xi\rangle}\right)
\end{equation}
is the Gaussian density with covariance matrix
\begin{equation}\label{fg6}
\La=C-B^*A^{-1}B =\left(C^{q}_{q'} - B_{q} A^{-1} B_{q'}\right),
\;\;(q = 1, \dots, m)
\end{equation}
and
\begin{equation}\label{fg7}
Z(x)=\frac{\sqrt{\det\Lambda}}{\pi\sqrt{\det\De}} =\frac{1}{\pi
\sqrt{  A}}\,.
\end{equation}

In the case at hand,

\begin{eqnarray}
\Delta^{N}(x)&=&\left(
\begin{array}{cc}
A^{N} & B^{N}\\
B^{N *} & C^{N}
\end{array}\right)\,,\nonumber \\
\big( A^{N} \big) &=& \E\big ( X^2\big)=\frac{1}{d_N}\Pi_{I_N}(x,x)\,,\nonumber\\
\big( B^{N}\big)_q&=& \E\big( X\Xi_{q}\big)= \frac{1}{d_{N}}
\frac{\partial}{\partial y_q} \Pi_{I_N}(x,y)|_{x = y}\,,\nonumber\\
\big( C^{\lambda}\big)^{q}_{q'}&=& \E\big(  \Xi_{q} \Xi_{q'}\big)=
 \frac{1}{d_{N}}
\frac{\partial^2}{\partial x_q \partial y_{q'} }\Pi_{I_N}(x,y)|_{x = y}\,,\nonumber\\
&&  \quad q, q'=1,\dots,m\,.\nonumber
\end{eqnarray}

Making a simple change of variables in the integral (\ref{d2}), we
have

\begin{prop}\label{BSZ}  \cite{BSZ1} On a real Riemannian manifold of dimension $m$, the
density of zeros of a random Riemannian wave is

\begin{equation}\begin{array}{lll}\label{EVOL1}
K_1^N(x) & = &\frac{1 }{\pi^{m} (\sqrt{d_N^{-1} \;\;
\Pi_{I_N}(x,x)}} \int_{\R^m} || \Lambda^N (x)^{1/2} \xi||
\exp\left( - {\langle \xi,\xi\rangle}\right)d\xi,
  \end{array}
\end{equation}
where $\Lambda^N(x)$ is a symmetric form on $T_xM$. For the
asymptotically fixed freqency ensembles, it is  given by
$$\Lambda^N(x) = \frac{1}{d_N} \left(d_x \otimes d_y \Pi_{I_N}(x, y) |_{x = y} -
\frac{1}{\Pi_{I_N}(x, y)} d_x  \Pi_{I_N}(x, y) |_{x = y} \otimes
 d_y \Pi_{I_N}(x, y) |_{x = y} \right). $$
 In the cutoff ensemble the formula is the same except that
 $\Pi_{I_N}$ is replaced by $\Pi_{[0, N]}$.

\end{prop}

\section{Zeros of random real Riemannian waves: Proof of Theorem \ref{ZMEASURE}}

We begin the proof with the simplest case of the round metric on
$S^m$. Throughout this article, $C_m$ denotes a constant depending
only on the dimension. It may change from line to line.

\subsection{Random spherical harmonics}

To prove Theorem (\ref{ZDSM}) on a round $S^m$,  we first need to
evaluate the matrix above when $\Pi_N(z,w)$ is the orthogonal
projection onto spherical harmonics of degree $N$.

\begin{prop} \label{PINSM} Let $\Pi_N: L^2(S^m) \to \hcal_N$ be the orthogonal
projection. Then:

\begin{itemize}

\item (A) $\Pi_N(x,x) = \frac{1}{Vol(S^m)} d_N$;

\item (B) $d_x \Pi_N(x, y)|_{x = y} = d_y \Pi_N(x, y)|_{x = y} =
0$;

\item (C) $d_x \otimes d_y \Pi_N(x, y)|_{x = y} = \frac{1}{m
Vol(S^{m})} \lambda_N^2 d_N g_x.$

\end{itemize}

\end{prop}

\begin{proof}
Statement (B) follows from statement (A) since $$0 = d_x
\Pi_N(x,x) = d_x \Pi_N(x, y)|_{x = y} + d_y \Pi_N(x, y)|_{x = y}$$
and because $d_x \Pi_N(x, y)|_{x = y} = d_y \Pi_N(x, y)|_{x = y}$.
Statement (C) holds because $d_x \otimes d_y \Pi_N(x, y)|_{x = y}
= C_N g_x$ by $SO(m + 1)$ symmetry. To evaluate $C_N$ we use that
$$0 = \Delta \Pi_N(x, x) = 2 Tr d_x \otimes d_y \Pi_N(x, y)|_{x = y}
 - 2 \lambda_N^2 \Pi_N(x,x).$$
 Here, $Tr  d_x \otimes d_y \Pi_N(x, y)|_{x = y}$ denotes the contraction. It equals
 $ \sum_{j = 1}^{d_N} ||d\phi_{Nj} (x)||^2$ where $\{\phi_{N j}\}$ is an orthonormal basis. Thus,
 $m C_N = \lambda_N^2 \Pi_N(x,x)$ and the formula of (C) follows
 from (A).

\end{proof}

The expected density of random  nodal hypersurfaces is given as
follows

\begin{prop} \label{REALDENSITY} In the case of $S^m$,

\begin{equation}\begin{array}{lll}\label{EVOL}
K_1^N(x) & = & C_m  \lambda_N\sim C_m N ,  \end{array}
\end{equation} where $C_m = \frac{1}{\pi^{m}} \int_{\R^m} |\xi|
\exp\left( - {\langle\xi,\xi\rangle}\right) d\xi.$

\end{prop}

\begin{proof}

By Propositiosn \ref{BSZ} and \ref{PINSM},  we have
\begin{equation} K_1^N(x) = \frac{\sqrt{Vol(S^m)}}{\pi^{m}} \int_{\R^m} || \Lambda^N (x)^{1/2} \xi||
\exp\left( - {\langle \xi,\xi\rangle}\right)d\xi,
\end{equation}
where
$$\Lambda^N(x) = \frac{1}{d_N} \left(\frac{1}{m
Vol(S^{m})} \lambda_N^2 d_N g_x \right). $$

\end{proof}

\subsection{Random Riemannian waves: proof of Theorem \ref{ZMEASURE}}

We now generalize the result to any compact $C^{\infty}$
Riemannian manifold $(M, g)$ which is either aperiodic or Zoll.
As in the case of $S^m$, the key issue is the asymptotic behavior
of derivatives of the spectral projections
\begin{equation}\label{RESPECPROJ}  \Pi_{I_N}(x, y) = \sum_{j:
\lambda_j \in I_N} \phi_{\lambda_j}(x) \phi_{\lambda_j}(y).
\end{equation}

\begin{prop}\label{PING}   Assume $(M, g)$ is either aperiodic and $I_N = [N, N + 1]$ or Zoll and
 $I_N$ is a cluster decomposition. Let $\Pi_{I_N}: L^2(M) \to
\hcal_N$ be the orthogonal projection. Then:

\begin{itemize}

\item (A) $\Pi_{I_N}(x,x) = \frac{1}{Vol(M, g))} d_N (1 + o(1))$;

\item (B) $d_x \Pi_{I_N}(x, y)|_{x = y} = d_y \Pi_N(x, y)|_{x = y}
= o(N^{m })$;

\item (C) $d_x \otimes d_y \Pi_{I_N}(x, y)|_{x = y} =
\frac{1}{Vol(M, g))} \lambda_N^2  d_N g_x (1 + o(1)).$

\end{itemize}

In the aperiodic case,
\begin{enumerate}

\item $\Pi_{[0, \lambda]} (x,x) = C_m \lambda^m + o(\lambda^{m -
1}); $

\item $d_x \otimes d_y \Pi_{[0, \lambda]}(x, y) |_{x = y} = C_m
\lambda^{m +2} g_x + o(\lambda^{m + 1 }). $

\end{enumerate}
In the Zoll case, one adds the complete asymptotic expansions for
$\Pi_{I_N}$ over the $N$ clusters to obtain expansions for
$\Pi_N$.

\end{prop}

\begin{proof} Asymptotic formulae of type (A) are standard in spectral
asymptotics, and we refer to \cite{DG,Ho} for background.
Asymptotics of type (C) were worked out in \cite{Z2} (Theorem 2)
in the special case of a Zoll metric. However, much of the
calculation goes through for any compact Riemannian manifold. It
does not appear however that (B) has been stated before or that
(C) has been previously discussed on general Riemannian manifolds,
although the techniques are standard.  These asymptotics are dual
to the heat kernel asymptotics in \cite{BBG}, but they are sharper
because we are using spectral intervals for $\sqrt{\Delta}$ of
fixed width rather than intervals of the form $[0, \lambda]$,
which are dual to heat kernel asymptotics. Thus, we need two term
asymptotics for long intervals in order to obtain asymptotics on
short intervals.

We follow the standard Tauberian method of \cite{DG}, Proposition
2.1,  for studying the spectral asymptotics. We consider the
spectral measure
\begin{equation} \begin{array}{ll}   d_{\lambda} \Pi_{[0, \lambda]}(x, y)  = &  \sum_j \delta(\lambda -
 \lambda_j)  \phi_{\lambda_j}(x) \phi_{\lambda_j}(y), \end{array}
 \end{equation}
 whose integral over the interval $I_N$  equals $\Pi_{I_N}(x,y)$,
 and the derived measures
\begin{equation}\left\{ \begin{array}{ll}
 (a) &  d_{\lambda} \Pi_{[0, \lambda]}(x, x) = \sum_j \delta(\lambda -
 \lambda_j)  \phi_{\lambda_j}(x)^2\\ & \\
(b) &  d_{\lambda} d_x \Pi_{[0, \lambda]}(x, x) =  2 \sum_j
\delta(\lambda -
 \lambda_j)  \phi_{\lambda_j}(x) d \phi_{\lambda_j}(x) \\ &
 \\ (c) &  d_{\lambda} d_x \otimes d_y \Pi_{[0, \lambda]}(x, y) |_{x = y}  = \sum_j \delta(\lambda -
 \lambda_j)  d \phi_{\lambda_j}(x) \otimes d \phi_{\lambda_j}(x).
 \end{array} \right.
 . \end{equation}

 We now introduce  a cutoff function $\rho \in
\scal(\R)$ with $\hat{\rho} \in C_0^{\infty}$ supported in
sufficiently small neighborhood of $0$.  We also assume
$\hat{\rho} \equiv 1$ in a smaller neighborhood of $0$. Then there
exists an expansion in
  inverse powers of $\lambda$ with coefficients smooth in $(x,
  y)$:
\begin{equation}\label{EXPANSION} \rho * d_{\lambda} \Pi_{[0, \lambda]}(x, x) =  \sum_j
\rho(\lambda - \lambda_j) \phi_{\lambda_j}^2(x)
  \sim \sum_{k = 0}^{\infty}
\lambda^{m - 1 - k} \omega_k(x),
\end{equation}
where $\omega_k $ are smooth in $x$, and $\omega_0 = 1$.

We briefly recall the proof of (\ref{EXPANSION}) and of (A):  We
have
\begin{equation}\label{EXPANSIONa} \rho * d_{\lambda} \Pi_{[0, \lambda]}(x, y) =
\int_{\R} \hat{\rho}(t) e^{ i t \lambda} U(t, x, y) dt,
\end{equation}
where $U(t, x, y)$ is the Schwartz kernel of the wave group $U(t)
= e^{- i
 t \sqrt{\Delta}}$.
 We use a  small-time parametrix for  $U(t, x, y)$ near the diagonal  of the form
\begin{equation} \label{PARAONE} U(t, x, y) = \int_{T^*_y M} e^{- i
t |\xi|_{g_y} } e^{i \langle \xi, \exp_y^{-1} (x) \rangle} A(t, x,
y, \xi) d\xi
\end{equation} where $|\xi|_{g_x} $ is the metric norm function at
$x$, and where $A(t, x, y, \xi)$ is a polyhomogeneous amplitude of
order $0$ which is supported near the diagonal. Setting  $x = y$
gives
\begin{equation}\label{EXPANSIONb} \rho * d_{\lambda} \Pi_{[0, \lambda]}(x, x) =
\int_{\R}  \int_{T^*_y M} \hat{\rho}(t) e^{ i t \lambda}  e^{- i t
|\xi|_{g_y} } A(t, x, x, \xi) d\xi dt.
\end{equation}
As in \cite{DG}, we  pass to polar coordinates $r = |\xi|_g$,
change variables $\theta \to \lambda \theta$ and apply the
stationary phase method to the $dr dt $ integral to obtain
(\ref{EXPANSION}).

For (C),  we apply the method of \cite{Z2} (see (3.6) - (3.7)). We
denote the phase of $U(t, x, y)$ by
$$\phi(t, x, y, \xi) =
 \langle \xi, \exp_y^{-1} (x) \rangle - t |\xi|_{g_y}. $$
Then applying
 $d_x \otimes d_y |_{x = y}$ to the integral produces a highest
 order term given by a universal constant times
 $$ d_x \phi (t, x, y,  \xi) \otimes d_y\; \phi (t, x, y,  \xi) |_{x = y} =   \;\;  \xi \otimes
 \xi, $$
 since $a_0(t, x, x, \xi) = 1$ (cf. \cite{DG}). It also
 produces lower order terms (i.e. of order $\leq  1$)  in which at least one derivative falls
 on the amplitude. If we the put the $d\xi$- integral in polar
 coordinates $\xi = r \omega$, we obtain an expansion
 \begin{equation}\label{EXPANSIONc} \rho * d_{\lambda}\; d_x \otimes d_y \Pi_{[0, \lambda]}(x, y) |_{x = y}
  \sim \sum_{k = 0}^{\infty}
\lambda^{m + 1 - k} B_k(x),
\end{equation}  with the leading coefficient
 \begin{equation} B_0 (x) = \frac{C_m}{Vol(M, g)} \; \int_{S^*_x
 M}  \omega \otimes \omega d \mu_x(\omega) =  \frac{C_m}{Vol(M, g)} g_x.   \end{equation}
Here, $d\mu_x$ is the Euclidean area element induced by $g$ on
$S^*_x M$.

We now draw the conclusions for (A) - (C). In the Zoll case, (A)
and (C)  are already proved in detail in \cite{Z2}. In the Zoll
case, there exist complete asympotic expansions for the spectral
sums and we may deduce (B) as well from the smoothed expansion by
a modification of the proof of \cite{Z2}, Theorem 2.  In the
aperiodic case, the measures (A) and (C) are positive and we may
apply  the Fourier Tauberian theorems of \cite{Ho,SV} (see the
Appendix in \S \ref{APPENDIX}), or alternatively the  remainder
estimate of Ivrii, to obtain two-term expansions:
\begin{enumerate}

\item $\Pi_{[0, \lambda]} (x,x) = C_m \lambda^m + o(\lambda^{m -
1}); $

\item $d_x \otimes d_y \Pi_{[0, \lambda]}(x, y) |_{x = y} = C_m
\lambda^{m +2} g_x + o(\lambda^{m + 1 }). $

\end{enumerate}

We then subtract the expansions across the interval $I_N$ to
obtain the stated result. We note that the drop in degree is
encoded in $d_N \sim N^{m-1}$.

We still need to analyze (B). Since (B) equals  $d_x
\Pi_{I_N}(x,x)$, we could take the derivative in $x$ of
(\ref{EXPANSIONb})  to obtain
\begin{equation}\label{EXPANSIONbc} \rho * d_{\lambda}  d_x \Pi_{[0, \lambda]}(x, x) =
\int_{\R}  \int_{T^*_y M} \hat{\rho}(t) e^{ i t \lambda}  e^{- i t
|\xi|_{g_y} } d_x A(t, x, x, \xi) d\xi dt
\end{equation}
We then have an expansion similar to that of (\ref{EXPANSION})
except that the amplitude is now the one-form $d_x A$. However,
the  Tauberian theorems do not apply  to  (B) since $d_{\lambda}
d_x \Pi_{[0, \lambda]}(x,x)$  is not a positive measure. In the
Zoll case, we have a complete asymptotic expansion of (A) and its
$x$-derivative gives that of (B). But on a general Riemannian
manifold one cannot use this approach.

Henceforth we assume $(M, g)$ is aperiodic. In this case, we  use
the fact that,  for each $k$, $B_k = \frac{\partial}{\partial x_k}
\Delta^{-1/2} $ is a bounded pseudo-differential operator and
$$\frac{\partial}{\partial x_k}  \Pi_{[0, \lambda]}(x,y)|_{x = y}  = \sum_{j: \lambda_j
\leq \lambda} \lambda_j \left( B_k \phi_{\lambda_j}(x) \right)
\phi_{\lambda_j}(x).
$$
We write
$$2 \left( B_k \phi_{\lambda_j}(x) \right) \phi_{\lambda_j}(x) = \left( (I + B_k) \phi_{\lambda_j}(x)\right)^2 - \left(B_k
\phi_{\lambda_j}(x)\right)^2 - \phi_{\lambda_j}^2. $$

We then substitute the right side into the summatory function in
$\lambda_j$ to obtain three asymptotic expansions to which the
Tauberian theorems apply. We start with
\begin{equation}\begin{array}{l}  \rho * d_{\lambda}  ((I + B_k)_x \otimes (I + B_k)_y
\sqrt{\Delta}    \Pi_{[0, \lambda]}(x, y)_{x = y} \\ \\ = ((I +
B_k)_x \otimes (I + B_k)_y \int_{\R} \hat{\rho}(t) e^{- i t \lambda} \sqrt{\Delta} U(t, x, y) |_{x = y} dt \\
\\
 = ((I + B_k)_x \otimes (I + B_k)_y \int_{\R} \hat{\rho}(t) e^{
i t \lambda} \frac{\partial}{i \partial t} U(t, x, y) |_{x = y}
\\ \\
= ((I + B_k)_x \otimes (I + B_k)_y  \left( \int_{\R} \int_{T^*_y
M} (i \hat{\rho}' (t) -   \lambda \hat{\rho})  e^{ i t \lambda}
e^{- i t |\xi|_{g_y} } e^{i \langle \xi, \exp_x^{-1}(y) \rangle}
A(t, x, x, \xi) d\xi dt \right)|_{x = y}
\end{array}
\end{equation}
To leading order, application of  $((I + B_k)_x \otimes (I +
B_k)_y $ under the integration sign multiplies the leading term of
the amplitude by  $(1 + b_k)(x, d\phi_t)$ where $\phi_t$ is the
phase (cf. e.g. the `fundamental asymptotic expansion' of
\cite{T}). We then apply the stationary phase method and due to
the extra factor of $\lambda$ in the amplitude obtain

\begin{equation} \begin{array}{l}
\sum_{j: \lambda_j \leq \lambda} \lambda_j \left( (I + B_k)
\phi_{\lambda_j}(x)\right)^2  = C_m \; \lambda^{m + 1} \int_{S^*_x
M} (1 + b_k(x, \omega))^2 d\mu(\omega) + o(\lambda^{m}).
\end{array}
\end{equation}
The $o(\lambda^m)$ remainder holds as in the scalar case because
the geodesic flow is aperiodic \cite{DG,Ho,SV}.  We then repeat
the calculation for $B_k$ and for $I$ and subtract. We clearly
cancel the leading term, leaving the remainder $o(\lambda^m)$.
When we subtract the interval $[0, N]$ from $[0, N + 1]$ we
obtain $o(N^m)$.

\end{proof}

\subsection{Proof of Theorem \ref{ZMEASURE} }

The generalization of Proposition \ref{REALDENSITY} to a general
Riemannian manifold is the following:

\begin{prop} \label{REALDENSITYg} For the asymptotically fixed frequency
ensemble, and for any $C^{\infty}\;\; (M, g)$ which is either Zoll
or aperiodic (and with $I_N$ as in Proposition \ref{PING}) , we have

\begin{equation}\begin{array}{lll}\label{EVOLa}
K_1^N(x)  & = & \frac{1}{\pi^{m} (\lambda_N)^{m/2}  } \int_{\R^m}
||\xi|| \exp\left( -\frac{1}{ \lambda_N}
{\langle\xi,\xi\rangle}\right) d\xi + o(1) \\& &  \\
& \sim &  C_m N ,  \end{array}
\end{equation} where $C_m = \frac{1}{\pi^{m}} \int_{\R^m} ||\xi||
\exp\left( - {\langle\xi,\xi\rangle}\right) d\xi.$ The same
formula holds for the cutoff ensemble.

\end{prop}

\begin{proof}

Both on a sphere $S^m$ or on a more general $(M, g)$ which is
either Zoll or aperiodic, we have by Propositions \ref{PINSM}
resp. \ref{PING} and the general formula for $\Delta^N$ in \S
\ref{DEN} that

\begin{eqnarray}
\; \Delta^N(z)&=& \frac{1}{Vol(M, g)}  \left(
\begin{array}{cc}
(1 + o(1))   & o(1)\\
o(1) &  N^2  \;  g_x(1 +   o(1))
\end{array}\right)\,,\nonumber \\
\end{eqnarray}
It follows that
\begin{equation}\label{fg6a}
\La^N=C^N-B^{N *}(A^N)^{-1}B^N =  \frac{1}{Vol(M, g)}  N^2 \; g_x +
o(N). \end{equation}

Thus, we have
\begin{equation} \begin{array}{lll} K_1^N(x) & \sim  &
 \frac{\sqrt{Vol(M, g)}}{\pi^{m} } \int_{\R^m} || \Lambda^N (x)^{1/2} \xi||
\exp\left( - {\langle \xi,\xi\rangle}\right)d\xi \\ && \\
& = & \frac{N}{\pi^{m}} \int_{\R^m} || (I + o(1)) (x)^{1/2} \xi||
\exp\left( - {\langle \xi,\xi\rangle}\right)d\xi,
\end{array} \end{equation}
where $o(1)$ denotes a matrix whose norm is $o(1)$. The integral
tends to $\int_{\R^m} ||   \xi|| \exp\left( - {\langle
\xi,\xi\rangle}\right)d\xi$ as $N \to \infty$, completing the
proof.

\end{proof}

So far, we have only determined the expected values of the nodal
hypersurface measures. To complete the proof of Theorem
\ref{ZMEASURE}, we need to prove:

\begin{prop} \label{VAR} If $(M, g)$ is real analytic, then
 the variance of $\frac{1}{\lambda_N} X^N_{\psi}$ is bounded.
 \end{prop}

 \begin{proof} By  a theorem of   Donnelly-Fefferman \cite{DF},
 for real analytic $(M, g)$,
\begin{equation} \label{DF} c_1 \lambda \leq {\mathcal
H}^{m-1}(Z_{\phi_{\lambda}}) \leq C_2 \lambda, \;\;\;\;\;\;(\Delta
\phi_{\lambda} = \lambda^2 \phi_{\lambda}; c_1, C_2 > 0).
\end{equation}
Hence for any  $f_N \in \hcal_{I_N}$,  $\frac{1}{\lambda_N}
Z_{f_N}$ has bounded mass. Hence, the random variable
$\frac{1}{\lambda_N} X^N_{\psi}$ is bounded, and therefore so is
its variance.

\end{proof}

\begin{rem}

The variance of $\frac{1}{\lambda_N} X^N_{\psi}$ is given by
\begin{equation} Var (\frac{1}{\lambda_N} X^N_{\psi}) =
\frac{1}{\lambda_N^2} \int_M \int_M \left(K^N_2(x, y) - K^N_1(x)
K^N_1(y) \right) \psi(x) \psi(y) dV_g(x) dV_g(y), \end{equation}
where $K_2^N(x, y) = \E_{\gamma_N} (Z_{f_N}(x) \otimes
Z_{f_N}(y))$ is the pair correlation function for zeros.  Hence,
boundedness would follow from
\begin{equation}
\frac{1}{\lambda_N^2} \int_M \int_M  K^N_2(x, y) \; dV_g(x)
dV_g(y) \leq C.
\end{equation}
There is a formula similar to that for the density in Proposition
\ref{BSZ} for $K^N_2(x,y)$ and it is likely that  it could be
used to prove boundedness of the variance for any $C^{\infty}$
Riemannian manifold. But we leave this for the future. In the
\kahler case, asymptotic formulae for the variance of smooth
linear statistics are given in \cite{SZ2,SZ3}, but the method does
not apply in the real case.
\end{rem}

\subsection{Random sequences and  proof of Corollary \ref{ZDSM}}

We recall that the set of   random sequences  of Riemannian waves
of increasing frequency is  the probability space $\hcal_{\infty}
= \prod_{N=1}^{\infty} \hcal_{I_N}$ with the measure
$\gamma_{\infty}  = \prod_{N=1}^{\infty} \gamma_N$. An element in
$\hcal_{\infty}$ will be denoted ${\bf f} = \{f_N\}$. We have,
$$|(\frac{1}{\lambda_N} Z_{f_N}, \psi)|\leq \frac{1}{\lambda_N}
\hcal^{n-1}(Z_{f_N}) \; \|\psi\|_{{ C}^0}. $$ By a density
argument it suffices to prove that the linear statistics
$\frac{1}{\lambda_N} (Z_{f_N}, \psi) - \frac{1}{Vol(M, g)} \int_M
\psi dV_g \to 0$ almost surely in $\hcal_{\infty}$. From Theorem
\ref{ZMEASURE}, we have:

\begin{cor}

\noindent (i) $\lim_{N\rightarrow \infty} \frac{1}{N} \sum_{k\leq
N} \E(\frac{1}{\lambda_k} X^k_{\psi} )= \frac{1}{Vol(M, g)} \int_M
\psi dV_g;$

\noindent(ii)  $Var(\frac{1}{\lambda_N} X^N_{\psi})$
 is bounded on $\hcal_{\infty}$.

\end{cor}

  Since $\frac{1}{\lambda_N} X^N_{\psi}$ for
$\{,N=1,2,\ldots\}$ is a sequence of independent random variables
in $\hcal_{\infty}$ with bounded variances, the Kolmogorov strong
law of large numbers  implies that
$$ \lim_{N\rightarrow \infty} \frac{1}{N} \sum_{k\leq
N} (\frac{1}{\lambda_k} X^k_{\psi} )= \frac{1}{Vol(M, g)} \int_M
\psi dV_g$$ almost surely.

\section{Analytic continuation of eigenfunctions and spectral projections}

We now turn to complex zeros of analytic continuations of random
Riemannian waves. Before getting into the statistics of zeros, we
need to recall the basic results on analytic continuation of
eigenfunctions and to introduce the basic two-point kernels.

 As mentioned in the
introduction,  for each analytic metric $g$ there exists a unique
plurisubharmonic exhaustion function $\rho$ on $M_{\C}$ inducing a
\kahler metric $\omega_g = i \ddbar \rho$ which agrees with $g$
along $M$.
 We
recall $\sqrt{\rho}(\zeta) = \frac{1}{2 i} r_{\C}(\zeta,
\bar{\zeta})$ where $r(x,y)$ is the distance function and $r_{\C}$
is its holomorphic extension to a small neighborhood of the
anti-diagonal $(\zeta, \bar{\zeta})$ in $M_{\C} \times M_{\C}$;
$\sqrt{\rho}$ is a solution of the homogeneous complex
Monge-Amp\`ere equation $(\ddbar \sqrt{\rho})^m = 0$ away from the
real points.
 We refer to \cite{GS1, GS2, LS1, GLS,Z3} for further background on
Grauert tubes and adapted complex structures on cotangent bundles
of analytic Riemannian manifolds.

The eigenfunctions $\phi_j$ admit holomorphic extensions
$\phi_j^{\C}$ to the maximal Grauert tube \cite{Bou, GS2}. As a
result, one can holomorphically extend the spectral measures
$d\Pi_{[0, \lambda]}(x,y) = \sum_j \delta(\lambda - \lambda_j)
\phi_j(x) \phi_j(y) $ of $\sqrt{\Delta}$.
 The  complexified
diagonal  spectral projections
 measure is defined by
 \begin{equation} d_{\lambda} \Pi_{[0, \lambda]}^{\C}(\zeta, \bar{\zeta}) = \sum_j \delta(\lambda -
 \lambda_j) |\phi_j^{\C}(\zeta)|^2. \end{equation}
 Henceforth, we generally omit the superscript and write the
 kernel as $\Pi_{[0, \lambda]}^{\C}(\zeta, \bar{\zeta})$.
 This kernel is not a tempered distribution due to the exponential
 growth of $|\phi_j^{\C}(\zeta)|^2$. Since many asymptotic techniques
 assume spectral functions are of polynomial growth,  we simultaneously
 consider the damped spectral projections measure
  \begin{equation} \label{SPPROJDAMPED} d_{\lambda} P_{[0, \lambda]
  }^{\tau}(\zeta, \bar{\zeta}) = \sum_j \delta(\lambda -
 \lambda_j) e^{- 2 \tau \lambda_j} |\phi_j^{\C}(\zeta)|^2, \end{equation}
 which   is a temperate distribution as long as $\sqrt{\rho}(\zeta)
 \leq \tau. $ When we set $\tau = \sqrt{\rho}(\zeta)$ we omit the
 $\tau$ and put
  \begin{equation} \label{SPPROJDAMPEDz} d_{\lambda} P_{[0, \lambda]
  }(\zeta, \bar{\zeta}) = \sum_j \delta(\lambda -
 \lambda_j) e^{- 2 \sqrt{\rho}(\zeta) \lambda_j} |\phi_j^{\C}(\zeta)|^2. \end{equation}

The integral of the spectral measure over an interval $I$  gives
$\Pi_{I}(x,y) = \sum_{j: \lambda_j \in I} \phi_j(x) \phi_j(y)$.
Its complexification gives the kernel (\ref{CXSPECPROJ}) along the
diagonal,
\begin{equation}\label{CXSP} \Pi_{I}(\zeta, \bar{\zeta}) =
 \sum_{j: \lambda_j \in I}
|\phi_j^{\C}(\zeta)|^2,  \end{equation}  and the integral of
(\ref{SPPROJDAMPED}) gives its temperate version
\begin{equation}\label{CXDSP}  P^{\tau}_{I}(\zeta, \bar{\zeta}) =
 \sum_{j: \lambda_j \in I}  e^{- 2 \tau \lambda_j}
|\phi_j^{\C}(\zeta)|^2,
\end{equation}
or in the crucial case of $\tau = \sqrt{\rho}(\zeta)$,
\begin{equation}\label{CXDSPa}  P_{I}(\zeta, \bar{\zeta}) =
 \sum_{j: \lambda_j \in I}  e^{- 2 \sqrt{\rho}(\zeta)\lambda_j}
|\phi_j^{\C}(\zeta)|^2,
\end{equation}

\subsection{Poisson operator as a complex Fourier integral
operator}

The damped spectral projection measure  $d_{\lambda} \; P_{[0,
\lambda]}^{\tau}(\zeta, \bar{\zeta})$ (\ref{SPPROJDAMPED}) is dual
under the real Fourier transform in the $t$ variable to the
restriction
\begin{equation} \label{CXWVGP} U (t + 2 i \tau, \zeta, \bar{\zeta}) = \sum_j
e^{(- 2 \tau + i t) \lambda_j} |\phi_j^{\C}(\zeta)|^2
\end{equation}  to the
anti-diagonal of the mixed Poisson-wave group. We recall  that the
kernel   $U (t + i \tau, x, y)$ of the the Poisson-wave operator
$e^{i (t + i \tau) \sqrt{\Delta}}$ admits a holomorphic extension
in the $x$ variable to the closed Grauert tube $M_{\tau}$. The
adjoint of the Poisson kernel $U(i \tau, x, y)$ also admits an
anti-holomorphic extension in the $y$ variable. The sum
(\ref{CXWVGP}) are the diagonal values of the complexified wave
kernel
\begin{equation} \label{EFORM}
\begin{array}{lll} U (t + 2 i \tau, \zeta, \bar{\zeta}') &  = &
\int_M  U (t + i \tau, \zeta, y)  E(i \tau, y, \bar{\zeta}'
) dV_g(x)\\
&& \\
&&  = \sum_j  e^{(- 2 \tau + i t) \lambda_j} \phi_j^{\C}(\zeta)
\overline{\phi_j^{\C}(\zeta')}.
\end{array}
\end{equation} We obtain
(\ref{EFORM}) by  orthogonality of the real eigenfunctions on $M$.

Since $U(t + 2 i \tau, \zeta, y)$ takes its values in the CR
holomorphic functions on $\partial M _{\tau}$, we consider the
Sobolev spaces $\ocal^{s + \frac{n-1}{4}}(\partial M _{\tau})$ of
CR holomorphic functions on the boundaries of the strictly
pseudo-convex domaisn $M_{\epsilon}$, i.e.
$${\mathcal O}^{s + \frac{m-1}{4}}(\partial M_{\tau}) =
W^{s + \frac{m-1}{4}}(\partial M_{\tau}) \cap \ocal (\partial
M_{\tau}), $$ where $W_s$  is the $s$th Sobolev space and where $
\ocal (\partial M_{\epsilon})$ is the space of boundary values of
holomorphic functions. The inner product on $\ocal^0 (\partial M
_{\tau} )$ is with respect to the Liouville measure \begin{equation}
\label{LIOUVILLE} d\mu_{\tau} =
(i \ddbar \sqrt{\rho})^{m-1} \wedge d^c \sqrt{\rho}. \end{equation}

We then regard  $U(t + i \tau, \zeta, y)$ as the kernel of an
operator from $L^2(M) \to \ocal^0(\partial M_{\tau})$. It equals
its composition $ \Pi _{\tau} \circ U (t + i \tau)$
with the  \szego projector
$$ \Pi_{\tau} : L^2(\partial M_{\tau}) \to \ocal^0(
\partial M_{\tau})$$  for the tube $
M_{\tau}$, i.e.   the orthogonal projection onto boundary values
of holomorphic functions in the tube.

 This is a useful expression
for  the complexified wave kernel, because  $\tilde{\Pi}_{\tau}$
is a complex Fourier integral operator with a small wave front
relation. More precisely,  the real points of its  canonical
relation form  the graph $\Delta_{\Sigma}$ of the identity map on
the symplectic one
 $\Sigma_{\tau}
\subset T^*
\partial M_{\tau}$ spanned by the real one-form $d^c \rho$,
i.e. \begin{equation} \label{SIGMATAU} \Sigma_{\tau} = \{(\zeta; r
d^c \rho (\zeta)) ,\;\;\; \zeta \in \partial M_{\tau},\; r > 0 \}
 \subset T^* (\partial M_{\tau}).\;\;  \end{equation}
 We note that for each $\tau,$ there exists a symplectic equivalence $ \Sigma_{\tau} \simeq
 T^*M$ by the map $(\zeta, r d^c \rho(\zeta) )  \to
 (E_{\C}^{-1}(\zeta), r \alpha)$, where $\alpha = \xi \cdot dx$ is
 the action form (cf. \cite{GS2}).

The following result was first stated by  Boutet de Monvel (for
more details, see also \cite{GS2,Z3}).

\begin{theo}\label{BOUFIO}  \cite{Bou, GS2} \label{BDM} $\Pi_{\epsilon} \circ U (i \epsilon): L^2(M)
\to \ocal(\partial M_{\epsilon})$ is a  complex Fourier integral
operator of order $- \frac{m-1}{4}$  associated to the canonical
relation
$$\Gamma = \{(y, \eta, \iota_{\epsilon} (y, \eta) \} \subset T^*M \times \Sigma_{\epsilon}.$$
Moreover, for any $s$,
$$\Pi_{\epsilon} \circ U (i \epsilon): W^s(M) \to {\mathcal O}^{s +
\frac{m-1}{4}}(\partial  M_{\epsilon})$$ is a continuous
isomorphism.
\end{theo}

We obtain the  holomorphic extension of the eigenfunctions
$\phi_{\lambda}$ of eigenvalue $\lambda^2$  by applying the
complex Fourier integral operator $U(i \tau)$:
\begin{equation} U(i \tau) \phi_{\lambda} = e^{- \tau \lambda}
\phi_{\lambda}^{\C}. \end{equation}   Combining  Theorem
\ref{BOUFIO} with a stantard Sobolev estimate, we  obtain:

\begin{cor} \label{SOB}  \cite{Bou, GLS} Each eigenfunction $\phi_{\lambda}$ has a
holomorphic extension to $ B^*_{\epsilon} M$ satisfying
$$\sup_{\zeta \in M_{\epsilon}} |\phi^{\C}_{\lambda}(\zeta)| \leq
C_{\epsilon} \lambda^{m + 1} e^{\epsilon \lambda}. $$
\end{cor}

The order of magnitude reflects the fact that $E(i \tau)$ smooths
to order $- \frac{m-1}{4} $, since it is a bounded operator from a
manifold $M$ of real dimension $m$ to one $\partial M_{\tau}$ of
dimension $2m - 1$.

We will also need the following Lemma from \cite{Z3} (Lemma 3.1):

\begin{lem} \label{PSIDO} Let  $a \in S^0(T^*M-0)$.  Then for all $0 < \epsilon < \epsilon_0$,
we have: $$ U (i \epsilon)^* \Pi_{\epsilon} a
\Pi_{\epsilon} U(i \epsilon) \in \Psi^{-
\frac{m-1}{2}}(M),
$$
with principal symbol equal to $a(x, \xi) \; |\xi|_g^{- (
\frac{m-1}{2})}.$
\end{lem}

Here, $\Psi^s(M)$ is the class of pseudodifferential operators of
order $s$ and $\epsilon_0$ is the radius of the maximal Grauert
tube.

\section{Complex zeros of random waves:
Proof of Theorem \ref{ZERORAN}}

Our first result is a formula for the expected distribution of
complex zeros. We retain the same ensembles and notation from the
real case.

\begin{prop} \label{MAINPROP} For any test form $\psi$ ,
$$\frac{1}{\lambda} \langle \E_{\gamma_N}  [Z_{f_{N}^{\C}}], \psi \rangle = \frac{1}{N} \langle
 \ddbar \log \Pi_{I_N}(\zeta,
\bar{\zeta}), \psi \rangle + O(\frac{1}{N}).$$

\end{prop}

\begin{proof} By the Poincar\'e-Lelong formula
$Z_{f^{\C}} =\frac i{2\pi}  \ddbar \log |f^{\C}(\zeta)|^2$, hence
we have
\begin{equation} \E_{\gamma_N}  [Z_f^{\C}] =\frac i{2\pi}  \E_{\gamma_N} \ddbar \log
|f^{\C}(\zeta)|^2 = \frac i{2\pi}\ddbar \E_{\gamma_N} \log
|f^{\C}(\zeta)|^2 .
\end{equation}
 We write  $f = \sum_{\lambda_j \in I_N} c_{Nj}
\phi_{Nj}$ and then complexify to obtain $f^{\C}(\zeta) =
\sum_{\lambda_j \in I_N} c_{Nj} \phi_{Nj}^{\C}(\zeta)$. Thus, we
need to calculate $ \E_{\gamma_N} \log |\sum_{\lambda_j \in I_N}
c_{Nj} \phi_{Nj}^{\C}|$.

Define
$$\Phi_N = (\phi_{N1}, \dots, \phi_{N d_N}): M \to \C^{d_N}$$
to be the vector of eigenfunctions from the  orthonormal basis of
eigenfunctions in the spectral interval. Then,
$$|\Phi^{\C}_N(\zeta)|^2 = \sum_{j = 1}^{d_N} |\phi_{Nj}(\zeta)|^2
= \Pi_{I_N}^{\C}(\zeta, \bar{\zeta}). $$ We  use the notation:
\begin{equation} \frac{\Phi_N^{\C}(\zeta)}{|\Phi_N^{\C}(\zeta)|} = U(\zeta) +
i V(\zeta), ;\;\; U(\zeta), V(\zeta) \in \R^{d_N}, |U|^2 + |V|^2 =
1.
\end{equation}

\begin{lem}   \label{GKLEM} $ \E_{\gamma_N}\log |f^{\C}(\zeta)|^2 = \log \Pi_{I_N}(\zeta,
\bar{\zeta}) + G_N(\zeta, \bar{\zeta}), $  where $ G_N(\zeta,
\bar{\zeta})$ is the uniformly bounded sequence of
continuous functions on $M_{\C}$ given by

$$\begin{array}{l} G_N(\zeta, \bar{\zeta}) = \Gamma' (\frac{
1}{2}) + \Gamma (\frac{ 1}{2}) \log \max \{||U + J V||^2, ||U - J
V||^2\}, \end{array}$$
where $J$ is the natural complex structure on the real $2$-plane spanned by $U, V$ in $\R^{d_N}$.

\end{lem}

\begin{proof}

We have,

$$\begin{array}{lll} \E_{\gamma_N}  \log |\sum_{\lambda_j \in I_N} c_{Nj}
\phi_{Nj}^{\C}| & = &  \int_{\R^{d_N}} \log |\langle c,
\Phi_N^{\C}
\rangle| d\nu_k(a) \\ && \\
& = & \log | \Phi_N^{\C}(\zeta)|^2 + G_N(\zeta, \bar{\zeta}),
\end{array}$$
where  \begin{equation} \begin{array}{lll} G_N(\zeta, \bar{\zeta})
& = & \int_{\R^{d_N + 1}} e^{- |c|^2/2} \log |\langle c,
\frac{\Phi_N^{\C}(\zeta)}{|\Phi_N^{\C}(\zeta)|} \rangle| dc
\end{array} \end{equation}
Hence it suffices to calculate $G_{N}(\zeta, \bar{\zeta})$.

The span of  $U, V$ is a real two-dimensional plane in $\R^{d_N}$ when $U, V$ are linearly independent; we will
consider the other case in the remark below.   We  rotate coordinates in
the integral  for each $\zeta$ so that
$$U = (\cos \theta \cos \phi) e_1, \;\;\; V = (\cos \theta \sin
\phi) e_1 + \sin \theta e_2, $$ where $\{e_j\}$ is the standard
basis of $\R^m$ and where  $\theta, \phi$ are real angles
depending on $\zeta$ and $k$. Then,
$$\langle a, \frac{\Phi_k^{\C}}{|\Phi_k^{\C}(\zeta)|} \rangle =
a_1 (\cos \theta \cos \phi) + i \cos \theta \sin \phi) + i a_2
\sin \theta. $$

After rotating coordinates in the integral, the integrand depends
only on $a_1, a_2$, so  we may integrate out the remaining
variables $a_3, \dots, a_{d_N}$ and obtain
$$\begin{array}{lll} G_N (\zeta, \bar{\zeta}) & = & \int_{\R^{2}} e^{- |a|^2/2} \log |\langle a,
\frac{\Phi_N^{\C}(\zeta)}{|\Phi_N^{\C}(\zeta)|} \rangle| da
\\ && \\ & = & \int_0^{\infty} e^{- \frac{r^2}{2}} \int_{S^1} \log |\langle r
\omega, \frac{\Phi_N^{\C}(\zeta)}{|\Phi_N^{\C}(\zeta)|} \rangle| r
dr d\omega = \\ && \\ & = &  \Gamma' (\frac{ 1}{2})  +
\Gamma(\frac{1}{2})  \int_{S^1} \log |\langle  \omega,
\frac{\Phi_N^{\C}(\zeta)}{|\Phi_N^{\C}(\zeta)|} \rangle|  d\omega
\\ && \\ & = &  \Gamma' (\frac{ 1}{2})  + \Gamma(\frac{1}{2})
G_2^N(\zeta, \bar{\zeta}),
\end{array}$$
where \begin{equation}\begin{array}{lll}  G_2^N (\zeta,
\bar{\zeta}) & = & \int_{S^1} \log |\omega_1 (\cos \theta \cos
\phi + i \cos \theta \sin \phi) + i \omega_2 \sin \theta| d\omega
\\ && \\
& = & \int_0^{2 \pi}  \log |\alpha e^{i u} + \beta e^{- i u}| du =
\int_0^{2 \pi}  \log |\alpha e^{i 2u} + \beta| du \\ \\
& = & \max\;  \{\log |\alpha|, \;
\log |\beta|\}
\end{array}
\end{equation}  with
$$2 \alpha = \langle U + i V, e_1 - i e_2 \rangle =  \cos \theta e^{i \phi} +   \sin \theta,\;\;  2 \beta =
 \langle U + i V, e_1 + i e_2 \rangle  =  \cos
\theta e^{i \phi} -   \sin \theta. $$

In the two dimensional real space spanned by $e_1, e_2$ we may
define a complex structure by $J e_1 = e_2, Je_2 = - e_1$ and then
$$G_2^N (\zeta, \bar{\zeta}) = \log \max \{||U + J V||^2, ||U - J
V||^2\}.$$
For any two vectors $U, V$ with $||U||^2 + ||V||^2 = 1$,   $$1 \leq   \max \{||U + J V||^2, ||U
- J V||^2\} =   \max \{1 + 2 \langle U, J V \rangle, 1 - 2 \langle U, J V \rangle \} \leq 3.$$
 Therefore the logarithm is bounded above and below and is
clearly continuous. This completes the proof of Lemma \ref{GKLEM}.
\end{proof}

To complete the proof of Proposition \ref{MAINPROP}, it suffices
to  integrate the $\ddbar$ by parts onto $f$ in the $G_N$ term,
and use that $G_N$ is uniformly bounded. \end{proof}

\begin{rem} When  $\zeta \in M$ we have    $V = 0$ and $||U|| = 1$ so the calculation is slightly
different. It is possible that there exist other $\zeta \in M_{\C}$ such that $U, V$ are linearly dependent.
When $V = 0$, we put
$U =  e_1$ and obtain the limiting case when $\phi = 0$,
$\langle a, \frac{\Phi_k^{\C}}{|\Phi_k^{\C}(\zeta)|} \rangle =
a_1. $
We   integrate out all but $a_1$
 and  get
$$\begin{array}{lll} G_N (\zeta, \bar{\zeta}) & = & \int_{\R} e^{- |a|^2/2} \log |a| da, \;\;\; \zeta \in M.
\end{array}$$
For sufficiently Grauert tubes, it is impossible that $U = 0$, but our analysis has not ruled out  that $U, V$ may be linearly
dependent, i.e. the case $\theta = 0$. But the calculation above goes through without change in this case.

\end{rem}

\section{Complexified spectral projections and wave group}

In view of Proposition \ref{MAINPROP}, the remaining step in the
proof of  Theorem \ref{ZERORAN} is to show:
\begin{prop} \label{LAST}
\begin{equation} \label{LASTSTEPa} \frac{1}{\lambda} \log \Pi_{[0, \lambda]}(\zeta,
\bar{\zeta}) \to \sqrt{\rho},  \end{equation} and
\begin{equation} \label{LASTSTEP} \frac{1}{N} \log \Pi_{I_N}(\zeta,
\bar{\zeta}) \to \sqrt{\rho}.  \end{equation}

\end{prop}

For both the
long $[0, \lambda]$ and short $I_N$ intervals, the upper bound is
rather simple to prove. The lower bound is more difficult,
primarily for the short intervals. As discussed
above, in the Zoll case, the eigenvalues lie in clusters of width
$N^{-1}$ around an arithmetic progression, so unless the intervals
$I_N$ are centered along this progression there need not exist any
eigenvalues in $I_N$ and the logarithm would equal $- \infty$. We
have by definition centered the intervals on the arithmetic
progression, but clearly this choice must play a role in the
proof.

\subsection{Proof of Proposition \ref{LAST}}

We first give a simple comparison between  the logarithmic asymptotics
of $\Pi_{I_N}$ to those of $P_{I_N}$ (cf. (\ref{CXDSPa}).

\begin{lem} \label{COMPARISON}

 For  any $\tau = \sqrt{\rho}(\zeta) > 0$, there
exists $C, c > 0$ so that

$$\frac{c}{\lambda} + 2 \sqrt{\rho}(\zeta)  + \frac{1}{\lambda}
\log P_{[\lambda, \lambda + 1]}(\zeta, \bar{\zeta})
 \leq \frac{1}{\lambda} \log \Pi_{[\lambda, \lambda + 1]}(\zeta,
\bar{\zeta}) \leq 2 \sqrt{\rho}(\zeta)  + \frac{1}{\lambda} \log
P_{[\lambda, \lambda + 1]}(\zeta, \bar{\zeta}) +
\frac{C}{\lambda}, $$ hence
$$ \frac{1}{N} \log  \Pi_{I_N}(\zeta, \bar{\zeta}) = 2 \sqrt{\rho}(\zeta) +
\frac{1}{N} \log  P_{I_N} (\zeta, \bar{\zeta}) + O(\frac{1}{N}),
$$ where the remainder is uniform.
\end{lem}

\begin{proof}
Since $  -1 \leq N - \lambda_{N j} \leq 1$ for $\lambda_{N j} \in
I_N$, we have
\begin{equation}\label{ULB} c  e^{ 2 N
\sqrt{\rho(\zeta)} } \sum_{j = 1}^{d_N} e^{- 2 \sqrt{\rho(\zeta)}
\lambda_j} |\phi_{\lambda_j}^{\C}(\zeta)|^2 \leq \Pi_{I_N}(\zeta,
\bar{\zeta}) \leq C e^{ 2 N \sqrt{\rho(\zeta)} } \sum_{j =
1}^{d_N} e^{- 2 \sqrt{\rho(\zeta)} \lambda_j}
|\phi_{\lambda_j}^{\C}(\zeta)|^2.
\end{equation}

\end{proof}

Next we observe that the  upper bound in both  the cutoff ensemble
\begin{equation} \label{LASTSTEPaa}\limsup_{\lambda \to \infty}  \frac{1}{\lambda} \log \Pi_{[0, \lambda]}(\zeta,
\bar{\zeta}) \leq \sqrt{\rho}  \end{equation} follows immediately
from Corollary \ref{SOB} and the fact (cf. Lemma \ref{COMPARISON})
that $\Pi_{[0, \lambda]} \leq e^{\lambda \sqrt{\rho}} P_{[0,
\lambda]}. $ It is similar but  simpler in the fixed frequency
ensemble.

Hence, it  remains to prove is the lower bounds. As mentioned
above, the global nature of the geodesic flow must play a role in
the proof. The route we choose  uses the minimal amount of
information necessary to obtain the result: we only use the
behavior of the spectral projections in the real domain to deduce
a lower bound in the complex domain. We follow the method of
\cite{SZ1}.

Put  $V_N(\zeta): = \frac{1}{N} \log \Pi_{[N, N+1]}(\zeta,
\bar{\zeta})$. Then $V_N$ is plurisubharmonic in
$M_{\epsilon}$, and we would like to show that $V_N \to
\sqrt{\rho}$  in $L^1(M_{\epsilon})$. If not, we can find a
subsequence $\{V_{N_k}\}$ with $\|V_{N_k} -
\sqrt{\rho}\|_{L^1(M_{\epsilon})} \geq \delta>0$. Since $\{V_N\}$
is uniformly bounded above,  a standard result on subharmonic
functions (see \cite[Theorem~4.1.9]{Ho}) implies that sequence
$\{V_{N_k}\}$ either converges uniformly to $-\infty$ on
$M_{\epsilon}$ or else has a subsequence which is convergent in
$L^1(M_{\epsilon})$. The first possibility cannot occur since
$V_N(x) \to 0$ on the totally real submanifold $M$.

Therefore there  exists a subsequence, which we continue to denote
by $\{V_{N_k}\}$, which converges in $L^1(M_{\epsilon})$ to some
$G \in L^1(M_{\epsilon}).$  By passing  to a further subsequence,
we may assume that $V_{N_k} \to G $ converges pointwise  a.e.  in
$M_{\epsilon}$.  Let $V(\zeta)  = \limsup_{k \rightarrow
\infty}V_{N_k}(z)\;\;\;\;\;\;{\rm (a.e)}\;$ and  let
$$V^*(z):= \limsup_{w \rightarrow z} V(w)$$ be the
upper-semicontinuous regularization of $V$. Then $V^*$ is
plurisubharmonic on $M_{\epsilon} $, $V^* \leq \sqrt{\rho}$ (a.e.)
and $V^* = V$ a.e.

Since $\|V_{N_k} - \sqrt{\rho} \|_{L^1(U')} \geq \delta>0$, we
know that $V^*\not \equiv \sqrt{\rho} $. Hence, for some $\epsilon
> 0$, the open set $ U_\epsilon = \{z\in M_{\epsilon}: V^* < \sqrt{\rho} - \epsilon\} $ is
non-empty. Let $U''$ be a non-empty, relatively compact, open
subset of $ U_\epsilon$. Then  by Hartogs' Lemma, there exists a
positive integer $K$ such that for $\zeta \in U'',\ k\geq K$,
$V_{N_k} \leq \sqrt{\rho} -\epsilon/2$, i.e.
\begin{equation} \Pi_{[N, N+1]}(\zeta,
\bar{\zeta})\leq e^{(\sqrt{\rho} - \epsilon) N_k }
,\;\;\;\;\;\zeta\in U'', \;\;k \geq K.
\end{equation} By Lemma \ref{COMPARISON}, it follows that
\begin{equation} \label{ESTU} P_N (\zeta,
\bar{\zeta})\leq e^{ - \epsilon N_k } ,\;\;\;\;\;\zeta\in U'',
\;\;k \geq K.
\end{equation}

We now let  $\chi \in C_0^{\infty}(U'')$ be a smooth bump function
which equals one on some ball $B \subset U''$. By (\ref{ESTU}), we
have
\begin{equation} \label{ESTUa}  \int_{\partial M_{\epsilon}} \chi P_N (\zeta,
\bar{\zeta}) d\mu \leq e^{ - \epsilon N_k }.
\end{equation}
We observe that
\begin{equation} \label{ESTUc} \begin{array}{lll}  \int_{\partial M_{\epsilon}} \chi P_N (\zeta,
\bar{\zeta}) d\mu & = & \sum_{j: \lambda_j \in [N, N+1]} e^{- 2 \epsilon
\lambda_j} \int_{\partial M_{\epsilon}} \chi
|\phi_{\lambda_j}^{\C}(\zeta)|^2 d\mu \\ && \\ &= & \sum_{j:
\lambda_j \in [N, N+1]} \langle \chi  U(i \sqrt{\rho}(\zeta))
\phi_{\lambda_j},  U(i \sqrt{\rho}(\zeta)) \phi_{\lambda_j}
\rangle_{L^2(\partial M_{\epsilon})} \\ && \\ &= & \sum_{j: \lambda_j \in
[N, N+1]} \langle U(i \sqrt{\rho}(\zeta))^* \chi  U(i
\sqrt{\rho}(\zeta)) \phi_{\lambda_j}, \phi_{\lambda_j}
\rangle_{L^2(M)} \\ && \\
& = & Tr \Pi_{I_N}  U(i \sqrt{\rho}(\zeta))^* \chi  U(i
\sqrt{\rho}(\zeta)).  \end{array}
\end{equation}
By Theorem  \ref{BOUFIO} and Lemma  \ref{PSIDO}, $U(i
\sqrt{\rho}(\zeta))^* \chi U(i \sqrt{\rho}(\zeta))$ is a
pseudo-differential operator with principal symbol equal to
$|\xi|^{- \frac{n-1}{2}} \chi(x, \xi)$.  For intervals $I_N$
satisfying our assumptions, it follows from a standard local Weyl
law \cite{Ho} (see also \cite{Z1} for further discussion in the
present context) that for a zeroth order pseudodifferential
operator $A$, and for a Laplacian which is either aperiodic or
Zoll,
\begin{equation} Tr \Pi_{I_N} A = C_m N^{m-1}  \int_{B^*M} \chi
d\mu + o(N^{m-1}). \end{equation} Putting $A = \sqrt{\Delta}^{
\frac{n-1}{2}}  \; U(i \sqrt{\rho}(\zeta))^* \chi  U(i
\sqrt{\rho}(\zeta))$
\begin{equation} Tr \Pi_{I_N}  U(i \sqrt{\rho}(\zeta))^* \chi  U(i
\sqrt{\rho}(\zeta)) \sim C_m N^{- \frac{m-1}{2}} \int_{B^*M} \chi
d\mu. \end{equation} The latter asymptotics contradict the
exponential decay of (\ref{ESTUa}).

\qed

\section{\label{APPENDIX} Appendix on Tauberian Theorems}

We record here the statements of the Tauberian theorems that we
use in the article. Our main reference is \cite{SV}, Appendix B
and we follow their notation.

We denote by $F_+$ the class of real-valued, monotone
nondecreasing functions $N(\lambda)$  of polynomial growth
supported on $\R_+$. The following  Tauberian theorem uses only
the singularity at $t = 0$ of $\widehat{dN}$ to obtain a one term
asymptotic of $N(\lambda)$ as $\lambda \to \infty$:
\begin{theo} \label{ET} Let $N \in F_+$ and let $\psi \in
\scal (\R)$ satisfy the conditions:  $\psi$ is  even,
$\psi(\lambda)
> 0$ for all $\lambda \in \R$,   $\hat{\psi} \in
C_0^{\infty}$, and $\hat{\psi}(0) = 1$. Then,
$$\psi * dN(\lambda) \leq A \lambda^{\nu} \implies |N(\lambda) - N *
\psi(\lambda)| \leq C A \lambda^{\nu}, $$ where $C$ is independent
of $A, \lambda$.
\end{theo}

\end{document}